\documentclass{amsart}
\usepackage{amsmath}
\usepackage{latexsym}
\usepackage{amssymb}
\usepackage{graphics}
 
 \newtheorem{rem}{Remark}[section]
\newtheorem{theorem}{Theorem}[section]
\newtheorem{corollary}{Corollary}[section]

\newtheorem{lemma}{Lemma}[section]

\newtheorem{definition}{Definition}[section]
\numberwithin{equation}{section}
\numberwithin{table}{section}

\def\ad#1{\begin{aligned}#1\end{aligned}}  \def\b#1{\mathbf{#1}}
\def\a#1{\begin{align*}#1\end{align*}} \def\an#1{\begin{align}#1\end{align}}
\def\e#1{\begin{equation}#1\end{equation}} 
\def\p#1{\begin{pmatrix}#1\end{pmatrix}} 
  \numberwithin{equation}{section}
\def\boxit#1{\vbox{\hrule height1pt \hbox{\vrule width1pt\kern1pt
     #1\kern1pt\vrule width1pt}\hrule height1pt}\ }

 \def\lab#1{\boxit{\small #1}\label{#1}}
  \def\mref#1{\boxit{\small #1}\ref{#1}}
 \def\meqref#1{\boxit{\small #1}\eqref{#1}}
    
\long\def\comment#1{}
  \def\lab#1{\label{#1}} \def\mref#1{\ref{#1}} \def\meqref#1{\eqref{#1}}

\def\half{\frac{1}{2}}

\begin{document}

\baselineskip=14pt
\title{
  On the optimal convergence rate of a Robin-Robin
      domain decomposition method }


\author{Wenbin Chen\address{School of Mathematical Sciences,
  Fudan University, Shanghai, \break
   200437, China, {\tt wbchen@fudan.edu.cn} }
 \and
 Xuejun XU\address{LSEC, Institute of Computational Mathematics,
  Academy of Mathematics and System Sciences,
  Chinese Academy of Sciences, P.O. Box 2719, Beijing, 100190, China,
 {\tt
xxj@lsec.cc.ac.cn} } \and
 Shangyou Zhang \address{Department of Mathematical Sciences,
   University of Delaware, \break
  Newark, Delaware 18716, U.S.A., {\tt
szhang@udel.edu}}}

\begin{abstract}
  In this work, we solve a long-standing open problem:
  Is it true
  that the convergence rate of the Lions' Robin-Robin nonoverlapping
  domain decomposition(DD) method can be constant,
   independent of the mesh size $h$?
  We closed this twenty-year old problem with
    a positive answer.
  Our theory is also verified by numerical tests.

\end{abstract}

\maketitle

 \vskip 15pt

\noindent{\bf AMS subject classifications.}
    { 65N30, 65M60.}

% \date{January 1, 2001 and, in revised form, June 22, 2001.}

     \vskip 15pt

\noindent{\bf Keywords.}\quad 
  Finite element, Robin-Robin domain decomposition method, convergence rate

\pagestyle{myheadings} \thispagestyle{plain} 
\markboth{W. Chen, X. Xu and S. Zhang}{Domain decomposition}

\section{Introduction}

 \def\udn{\frac{\partial u}{\partial \b n }}
 \def\wdn{\frac{\partial w}{\partial \b n }}

Domain decomposition (DD) methods are important tools for
   solving partial differential equations,
   especially by parallel computers.
In this paper, we shall
 study a class of nonoverlapping DD method, which is
  based on using Robin-Robin boundary conditions as transmission conditions on
  the subdomain interface. The idea of
  employing Robin-Robin coupling conditions in DD methods
   was first proposed by P.L. Lions in \cite{lions}.
In the past twenty years, there are many works on the analysis and
   applications of this DD method:
\comment{
  Despres \cite{Despres} (1991),
  Douglas and Huang \cite{DH,DH1} (1997, 1998),
  Deng \cite{deng2,deng3} (1997,2003),
  Du \cite{Du} (2001),
  Gander {\sl et al.} \cite{Gander2,Gander3} (2001,2002, 2003),
  Guo and Hou \cite{GH} (2003),
   Discacciati \cite{Discacciati} (2004),
  Flauraud and Nataf \cite{Nataf1} (2006),
  Gander \cite{Gander01, Gander02} (2006, 2007),
  Qin and Xu\cite{QX1,QX2,QX3} (2006, 2008),
  Discacciati {\sl et al.} \cite{DQV2007} (2007),
  Lui \cite{Lui1} (2009),
  Chen {\sl et al.} {\cite{chen1,chen2}} (2010, 2011). }
  Despres \cite{Despres},
  Douglas and Huang \cite{DH,DH1},
  Deng \cite{deng2,deng3},
  Du \cite{Du},
  Gander {\sl et al.} \cite{Gander2,Gander3},
  Guo and Hou \cite{GH},
   Discacciati \cite{Discacciati},
  Flauraud and Nataf \cite{Nataf1},
  Gander \cite{Gander01, Gander02},
  Qin and Xu \cite{QX1,QX2,QX3},
  Discacciati {\sl et al.} \cite{DQV2007},
  Lui \cite{Lui1}, and
  Chen {\sl et al.} \cite{chen1,chen2}.
We should say that the list is far from being complete.

%(cf. \cite{Achdou},
%\cite{BF},\cite{Despres}, \cite{DH}, \cite{DH1}, \cite{Feng},
%\cite{GH}, \cite{QV}, \cite{Gander1}, \cite{Gander2},
%\cite{Gander3}, \cite{Japhet}, \cite{Discacciati},\cite{deng2},
%\cite{deng3}, \cite{QX} and \cite{QX2} for details.)

% Gander et al. (), Deng(1997, 2003), Du(2001), Flauraud and Nataf(2006)
%Guo and Hou(2003), Qin and Xu(2006, 2007, 2010),Lui(2009, 2010) etc.

By comparison with other DD methods, Lions' DD method has several
  advantages. The iterative procedure is simple and
      much more highly parallel than others.
Because it employs Robin conditions, the method is specially suitable for
   solving Helmholtz and time-harmonic Maxwell equations.
There exists a lot of works in this direction,
   cf. \cite{Despres, BF, Gander3, DGG} for details.

Lions' Robin-Robin DD method was proposed in 1990 \cite{lions}, see Definition
    \mref{add} below (without Step 5).
The convergence (without any rate) is shown in \cite{lions,QV}.
Later, the convergence was improved to a geometric convergence
  \cite{DH1,DH,GH}, i.e, a rate of $1-O(h)$.
It was first pointed out by
  Gander, Halpern and Nataf in \cite{Gander2} that the optimal choice
  of relaxation parameter is $\gamma=O(h^{-1/2})$
   and the convergence rate $1-O(\sqrt h)$ could be achieved.
Recently, Xu and Qin \cite{XQ} give a rigorous
    analysis on this result and shows that the rate is
    asymptotically sharp.
However, without enough knowledge on the method,  the two parameters
   $\gamma_1$ and $\gamma_2$ in Lions' DD method are set equal,
   see Definition \mref{add} below, by researchers in   above
   references.
Thus, the rate of $1-O(\sqrt h)$ is generally believed optimal for the Lions'
  DD method.

This paper answers this long-standing open problem:
 Is it possible to achieve a rate of $1-C$ for some some constant $C>0$
   independent of the mesh size $h$?
We give a positive answer.
Yes, the constant rate of convergence is achieved by well-choosing three
   parameters in the Robin-Robin DD method, $\gamma_1$, $\gamma_2$ and
   $\theta$, in Definition \mref{add}.
Roughly speaking,  the optimal choices are
   \a{ \gamma_1 = O(1), \ \gamma_2=O(h^{-1}), \ \hbox{ and } \
     \theta = \frac {2t-1}{2t+1}, }
  where $t\approx 1$ is the ratio of spectral radii of two
   Dirichlet-Neumann operators on two subdomains.
It is shown in this paper, by three types of analysis, 
    that the error reduction rate of the DD method is optimal,
   $1-C$.

Next, we introduce
the Robin-Robin DD method through a simple model problem.
  We solve the following  model problem in 2D,
which is decomposed into two subproblems (cf. Figure \mref{two-d}):
\e{\lab{laplace} \left\{ \ad{
     -\Delta u & = f \quad &&\hbox{ in } \Omega_1, \\
      u & = 0 \quad &&\hbox{ on } \partial\Omega\cap
            \partial \Omega_1, \\
     u-w &= \udn - \wdn =0
       \quad &&\hbox{ on } \Gamma, \\
     -\Delta w & = f \quad &&\hbox{ in } \Omega_2, \\
      u & = 0 \quad &&\hbox{ on } \partial\Omega\cap
            \partial \Omega_2, } \right. }
where $\Gamma$ is an interface separating $\Omega_1$ and $\Omega_2$,
    and $\b n$ is an outward normal vector of $\Omega_1$ at $\Gamma$.
The DD method can be applied to  general elliptic
   PDEs, general domains and multiple subdomains, cf. \cite{deng2,QV}.

% 16 30 30 0 0
% 42 30 0 360 28
%  -1 0 0 0 0
%  0

 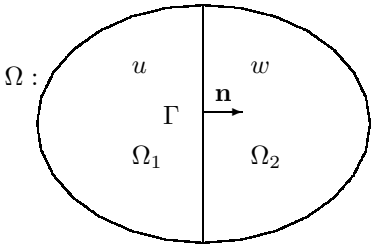
\begin{figure}[htb] \setlength{\unitlength}{1.5pt} \begin{center}
 \def\bd{\begin{picture}(  30.,  42.)(  0.,  0.)
     \def\la{\vrule width.4pt height.4pt}
  \multiput(  72.00,  30.00)( -0.026,  0.165){ 40}{\la}
  \multiput(  70.95,  36.67)( -0.073,  0.150){ 42}{\la}
  \multiput(  67.84,  43.01)( -0.110,  0.125){ 45}{\la}
  \multiput(  62.84,  48.70)( -0.136,  0.097){ 49}{\la}
  \multiput(  56.20,  53.45)( -0.152,  0.068){ 52}{\la}
  \multiput(  48.24,  57.02)( -0.162,  0.040){ 54}{\la}
  \multiput(  39.37,  59.24)( -0.166,  0.013){ 56}{\la}
  \multiput(  30.02,  60.00)( -0.166, -0.013){ 56}{\la}
  \multiput(  20.68,  59.25)( -0.162, -0.040){ 54}{\la}
  \multiput(  11.80,  57.04)( -0.152, -0.068){ 52}{\la}
  \multiput(   3.84,  53.47)( -0.136, -0.097){ 49}{\la}
  \multiput(  -2.81,  48.73)( -0.110, -0.125){ 45}{\la}
  \multiput(  -7.82,  43.04)( -0.074, -0.150){ 42}{\la}
  \multiput( -10.94,  36.71)( -0.026, -0.165){ 40}{\la}
  \multiput( -12.00,  30.03)(  0.026, -0.165){ 40}{\la}
  \multiput( -10.96,  23.36)(  0.073, -0.150){ 42}{\la}
  \multiput(  -7.86,  17.02)(  0.110, -0.125){ 45}{\la}
  \multiput(  -2.87,  11.33)(  0.135, -0.097){ 48}{\la}
  \multiput(   3.77,   6.57)(  0.152, -0.068){ 52}{\la}
  \multiput(  11.72,   2.99)(  0.162, -0.041){ 54}{\la}
  \multiput(  20.59,   0.76)(  0.166, -0.014){ 56}{\la}
  \multiput(  29.93,   0.00)(  0.166,  0.013){ 56}{\la}
  \multiput(  39.27,   0.74)(  0.162,  0.040){ 54}{\la}
  \multiput(  48.15,   2.95)(  0.152,  0.068){ 52}{\la}
  \multiput(  56.12,   6.51)(  0.136,  0.097){ 49}{\la}
  \multiput(  62.78,  11.25)(  0.110,  0.125){ 45}{\la}
  \multiput(  67.80,  16.93)(  0.074,  0.149){ 42}{\la}
  \multiput(  70.93,  23.26)(  0.026,  0.165){ 40}{\la}
 \end{picture}}
    \begin{picture}(60,60)(0,0)
    %  \multiput(0,0)(0,60){2}{\line(1,0){60}}
    % \multiput(0,0)(30,0){3}{\line(0,1){60}}
    \put(30,0){\line(0,1){60}}
    \put(0,0){\bd}
     \put(12,20){$\Omega_1$}  \put(42,20){$\Omega_2$}
     \put(12,43){$u$}  \put(42,43){$w$}
     \put(20,30){$\Gamma$}  \put(30,33){\vector(1,0){10}}
       \put(33,36){$\b n$}
    \put(-20,40){$\Omega:$}
 \end{picture}\end{center}
\caption{A domain is decomposed into two subdomains. }
\lab{two-d}
\end{figure}
  \def\udn{\frac{\partial u}{\partial \b n }}
 \def\wdn{\frac{\partial w}{\partial \b n }}

The Dirichlet and Neumann interface conditions on $\Gamma$
   in \meqref{laplace} are
  combined into two Robin interface conditions:
   \an{\lab{i1}
       \gamma_1  u +  \udn &= \gamma_1  w + \wdn =g_1
                        &&\hbox{ on } \Gamma, \\
      \lab{i2}   \gamma_2  u - \udn &=\gamma_2 w -  \wdn=g_2
             &&\hbox{ on } \Gamma.}
 Here we allow $\gamma_1, \gamma_2$ to be  any positive constants.
 For example, when $\gamma_1$ is arbitrarily close to   zero and
	$\gamma_2$ is close to
    infinity (but the linear systems would become near singular),
  the method would be reduced to the Dirichlet-Neumann DD method.
 The past researchers all set $\gamma_1=\gamma_2=\gamma$ in
   the Robin interface conditions, i.e., the two parameters are 
    simultaneously large or small.
 By selecting two parameters correctly, using the
   original Lions' DD method,
  this Robin-Robin domain decomposition method should be better
   than all existing  Dirichlet-Neumann,
     Neumann-Neumann and  Robin-Robin domain decomposition
    methods.

 Let $V_i=H_0^1(\Omega)|_{\Omega_i}$.  Later, $V_i$ also denotes
  the restriction of the
     finite element space of grid size $h$ on the two subdomains
     $\Omega_i$.
  By \meqref{i1}, we do an integration by parts on $\Omega_1$ to get
   \a{ \int_\Gamma  g_1 v ds
       & =  \int_\Gamma  (  \udn + \gamma_1 u)  v ds
          =  \int_{\Omega_1} (\nabla u\cdot\nabla v+
      \Delta u v) d\b x + \gamma_1 \int_\Gamma u  v ds \\
    & =   \int_{\Omega_1} (\nabla u\cdot\nabla v -f v)
        d\b x + \gamma_1 \int_\Gamma u  v ds.  }
  Thus
 \a{ a_1( u, v) + \gamma_1 \langle u,v\rangle
       & =  (f, v)_{\Omega_1} +  \langle g_1,v\rangle
    \quad \forall v\in V_1,
    }
  where \a{ a_i(u,v) &= \int_{\Omega_i} \nabla u \cdot\nabla v d\b x,
    \quad i=1,2, \\
      (f, v)_{\Omega_i} &= \int_{\Omega_i} f  v d\b x,
    \quad i=1,2, \\
        \langle u,v\rangle  &= \int_{\Gamma} u v ds. }
 Similarly, by \meqref{i2} and an integration by parts on $\Omega_2$,
  it follows (noting that $\b n$ is an inward normal vector to $\Omega_2$) that
   \a{ \int_\Gamma  g_2 v ds
       & =  \int_\Gamma  ( \gamma_2 w - \wdn )  v ds
          =  \int_{\Omega_2} (\nabla w\cdot\nabla v+
      \Delta w v) d\b x + \gamma_2 \int_\Gamma w  v ds \\
    & v=  \int_{\Omega_2} (\nabla w \cdot\nabla v -f v)
        d\b x +  \gamma_2  \int_\Gamma w v ds. }
  This way, we get the second variational problem on $\Omega_2$:
   \a{ a_2(w,v) +   \gamma_2 \langle w ,v\rangle
        = (f, v)_{\Omega_2} +  \langle g_2,v\rangle
    \quad \forall v\in V_2.
   }

\begin{definition}\lab{add}
 (The Robin-Robin DD method.)
\quad Given $g^0_1(=0)$ on $\Gamma$,
    a serial version domain decomposition iteration
   consists the following five steps ($m=0,1,\dots$):
  \begin{enumerate}
  \item Solve on $\Omega_1$ for $u^m$:
      \an{\lab{s1}  a_1( u^m, v) + \gamma_1\langle u^m,v\rangle
       & =  (f, v)_{\Omega_1} +  \langle g_1^m,v\rangle
    \quad \forall v\in V_1.
    }
  \item Update the interface condition on $\Gamma$:
      \an{\lab{s2} g^{m}_2 = - g^m_1
    + \big(\gamma_2 +  \gamma_1  \big) u^m.
    }
  \item Solve on $\Omega_2$ for $w^m$:
      \an{\lab{s3}  a_2(w^m,v) +   \gamma_2 \langle w^m ,v\rangle
        = (f, v)_{\Omega_2} +  \langle g_2^m,v\rangle
    \quad \forall v\in V_2.
    }
  \item Update the other interface condition on $\Gamma$:
      \an{\lab{s4} \tilde g^{m}_1 = -   {g^m_2}
    + \big( \gamma_1 +  \gamma_2  \big) w^m.
    }
  \item Get the next iterate by a relaxation:
      \an{\lab{s5}   g^{m+1}_1 =  \theta g_1^m +
          (1-\theta) \tilde g^{m}_1 .
    }

  \end{enumerate}

\end{definition}

  The rest of paper is organized as follows.
 In section 2, we shall show that although the  Robin-Robin
   DD method cannot achieve the geometrical convergence rate at the
   continuous PDE level, but it does at the discrete level.
 In section 3, we shall give an explicit convergence rate of the
  DD method on uniform meshes.  In section 4, we shall extend our
  method to more general quasi-uniform meshes. Using the Dirichlet-to-Neumann
  operator, we shall prove that the Robin-Robin DD method
  is optimal. Finally, in the last section, we shall present some numerical
  results to support our theory.
 It is seen from our numerical implementation that this DD method is better
  than Dirichlet-Neumann DD method and one-parameter Robin-Robin
  DD method.

\section{A Von Neumann analysis}\label{sec:von}
In this section,  through a simple model problem,
 we shall show that for the new DD method it is not possible
 to get the geometrical convergence rate strictly less than one
 at the continuous level, but
  it is possible at the discrete level.

Let us assume that $\Omega_1=[-\pi,0]\times[0,\pi]$ and
  $\Omega_2=[0,\pi]\times[0,\pi]$,
and it is enough for us to assume that $f\equiv 0$ so that the
  true solutions of Equation
  \meqref{laplace} vanishes.
Now if $g_1= \hat{g}_1\sin ky$ on $\Gamma$, from  Equations
    \meqref{laplace} and \meqref{i1}, the solution on $\Omega_1$ is
\a{
   u = \hat{u}\sinh(k(x+1))\sin ky, \quad \mbox{where} \quad
   \hat{u}=\frac{\hat{g}_1}{\gamma_1\sinh k+k\cosh k}.
}
If $g_2= \hat{g}_2\sin ky$ on $\Gamma$, from
   Equations \meqref{laplace} and \meqref{i2},  the solution on $\Omega_2$ is
\a{
   w = \hat{w}\sinh(k(x-1))\sin ky, \quad \mbox{where} \quad
    \hat{w}=-\frac{\hat{g}_2}{\gamma_2\sinh k+k\cosh k}.
}
By Definition \mref{add},
    if the initial error is $g_1^m=\hat{g}_1^m\sin ky$ on $\Gamma$,
then
\a{   g_2^m= \hat{g}_2^m \sin ky, \quad \mbox{where} \quad  \hat{g}_2^m =
\hat{g}_1^m\left( \frac{\gamma_2+\gamma_1}{\gamma_1+k\coth k}-1\right) .
}
Then, by \meqref{s3} and \meqref{s4},
\a{
   \tilde{g}_1^m = \hat{\tilde{g}}_1^m \sin ky, \quad \mbox{where} \quad
   \hat{\tilde{g}}_1^m  =  \hat{g}_2^m \left(
      \frac{\gamma_2+\gamma_1}{\gamma_2+k\coth k}-1\right).
}
Finally, after the relaxation step \meqref{s5},
\a{
   g_1^{m+1} = \hat{g}_1^{m+1} \sin ky,
}
where
\a{
 \hat{g}_1^{m+1} =\theta  \hat{g}_1^m + (1-\theta)\hat{\tilde{g}}_1^m
          = \rho \hat{g}_1^m,
 }
  and the factor
  \e{
    \rho = \theta + (1-\theta)\left(
     \frac{\gamma_2+\gamma_1}{\gamma_2+k\coth k}-1\right)
  \left(\frac{\gamma_2+\gamma_1}{\gamma_1+k\coth k}-1\right). \label{Von:rho}}
Now for the fixed parameters $0<\theta<1$, $\gamma_1>0$
   and $\gamma_2>0$,  if $k$ tends to infinity, then
\e{
   \rho \approx 1 -2(1-\theta)\frac{\gamma_1+\gamma_2}{k}\to 0.
}
Therefore in the continuous case, it is impossible to get the convergence rate
    independent of the
   frequency (or the wave number) $k$.
  On the other hand, if $k$ is bounded by $1\le k \le K$, we may
  obtain the convergence rate $\rho$ , which is independent of $k$
	(but dependent on $K$),
   through choosing the
   three parameters $\gamma_1$, $\gamma_2$ and $\theta$.

\begin{lemma}\label{lem:theta} If $a$ and $b$ are two non-negative
      constants, then the function
\e{ \lab{rho-theta}
    \rho(\theta)= \max\left\{ |\theta - (1-\theta)a|,
    |\theta - (1-\theta)b| \right\}
  }
  attains the minimum value $\frac{|b-a|}{2+a+b}$ at
  $\theta_0=\frac{a+b}{2+a+b}$.
\end{lemma}

 \begin{figure}[htb] \setlength{\unitlength}{1.5pt} \begin{center}
     \begin{picture}(140,66)(-10,-5)
     \def\la{\vrule width.4pt height.4pt}
      \put(-10,0){\vector(1,0){120}}
    \put(0,-6){\vector(0,1){65}}
     \put(50,0){\line(1,1){50}}  \put(50,0){\line(-1,1){50}}
     \put(25,0){\line(-3,2){25}}  \put(25,0){\line(3,2){75}}
    \put(-6,18){$a$}\put(-6,47){$b$}\put(115,-2){$\theta$}
   \put(70,10){\vector(-1,0){10}}\put(72,8){$|\theta - (1-\theta)b|$}
   \put(78,44){\vector(1,0){13}}\put(38,42){$|\theta - (1-\theta)a|$}
   \put(38,15){$A$}
% \multiput(0,20)(4.5,-1.75){11}{\multiput(0,0)(0.045,-0.0175){40}{\la}}
  \end{picture}\end{center}
\caption{Graphs of $|\theta - (1-\theta)a|$ and
	$|\theta - (1-\theta)b|$,  $b> a\ge 0$. }
\lab{two-line}
\end{figure}
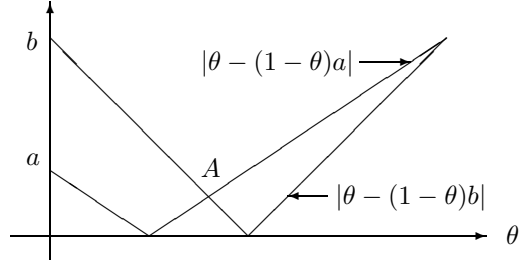

\begin{proof} Without loss of generality, we assume $b\ge a$.
   Both terms in \meqref{rho-theta} are
   piecewise linear functions.  We plot them in
	Figure \mref{two-line}.
   The minimal value is attained at the point $A$, where
   two lines intersect:
\a{
\theta - (1-\theta)a + \theta - (1-\theta)b=0.}
That is $\theta=\theta_0=\frac{a+b}{2+a+b}$, and
\a{
|\theta_0 - (1-\theta_0)a| = |\theta_0 - (1-\theta_0)b|=\frac{a+b}{2+a+b}.
}
So we get the lemma.
\end{proof}

\begin{lemma}\lab{lem:fx} For any $z\ge 0$, the function
\e{
 \omega(z) = \frac{\gamma_2-z}{\gamma_2+z}
	\cdot \frac{z-\gamma_1}{z+\gamma_1}
}
attains the maximum value at $z_0=\sqrt{\gamma_1\gamma_2}$:
\e{
   \max_{z>0}\omega(z)= \frac{(\eta-1)^2}{(\eta+1)^2},\quad
  \mbox{where}\quad \eta=\sqrt{\frac{\gamma_2}{\gamma_1}}.}
\end{lemma}

\begin{proof} The derivative of $\omega(z)$ is
 \an{\lab{w-d}
   \omega'(z) = \frac{2(\gamma_1+\gamma_2)
  (\gamma_1\gamma_2-z^2)}{(z+\gamma_1)^2(z+\gamma_2)^2}.
 }
 So $\omega(z)$ monotonically increases when $z<z_0$
     and monotonically decreases when $z>z_0$.
In particular,  then the minimum value of $\omega(z)$ on an
	interval $[z_1,z_2]$ is
     attained at one of the end points:
 \an{ \lab{end-min}
     \min_{z\in [z_1,z_2]}\omega(z) = \min\{ \omega(z_1),\omega(z_2)\}.
 }
  By \meqref{w-d}, $\omega(z)$ attains the only global
       maximum value at $z_0=\sqrt{\gamma_1\gamma_2}$:
\a{
    \omega(z_0)=\frac{\gamma_2-\sqrt{\gamma_1\gamma_2}}
    {\sqrt{\gamma_1\gamma_2}+\gamma_2}
  \frac{\sqrt{\gamma_1\gamma_2}-\gamma_1}
   {\sqrt{\gamma_1\gamma_2}+\gamma_1} =  \frac{(\eta-1)^2}{(\eta+1)^2}.
  }
The lemma is proved.
\end{proof}

\comment{
 g1=10; g2=100;
z=0:0.1:100; w=(g2-z)./(g2+z).*(z-g1)./(z+g1);
plot(z,(g2-z)./(g2+z).*(z-g1)./(z+g1));
z0=sqrt(g1*g2); et=sqrt(g2/g1);
[max(w) (et-1)^2/(et+1)^2]
}

From Equation \meqref{Von:rho},
  \a{
\rho  = \theta - (1-\theta)
     \frac{ \gamma_2-k\coth k}{\gamma_2+k\coth k} \cdot
     \frac{k\coth k-\gamma_1}{\gamma_1+k\coth k}
   =\theta-(1-\theta)\,\omega(k\coth k).
  }
If $\gamma_1$ and $\gamma_2$ are chosen such that
   $\gamma_1<\coth 1$ and  $\gamma_2>K\coth K$, then $\omega(k\coth k)>0$.
By  Lemma \ref{lem:fx},
 \a{
   |\rho|\le \max\{|\theta|, |\theta-(1-\theta)\omega(z_0)|\}.
 }
Applying Lemma \ref{lem:theta}, we may select 
 \e{\label{von:rho1o3} \theta_0=\frac{\omega(z_0)}{2+\omega(z_0)}
   \quad\Rightarrow \quad
   |\rho|\le |\theta_0| < \frac{1}{3}.
 }

 \begin{rem}\lab{r1}
 If $\eta>\frac{\sqrt{2}+1}{\sqrt{2}-1}$,
   we may just set $\theta=\frac{1}{3}$,
    and $|\rho|$ is also less than $\frac{1}{3}$.
 Moreover, this bound can be improved further
   if we carefully estimate the minimum value of $\omega(z)$.
 \end{rem}

 \begin{rem}\lab{rmk:gamma1} The constrain
   $\gamma_1<\coth1$ can be relaxed. Actually, if $\gamma_1>\coth 1$, then
  \a{
   |\rho|\le \max\left\{ \left|\theta+(1-\theta)\zeta\right|,
   |\theta-(1-\theta)\omega(z_0)|\right\},
 }
 where $\zeta=\frac{\gamma_1-\coth1}{\gamma_1+\coth1}$.
   Then we set $\theta=0$ if $\omega(z_0)\le \zeta$ and
 set $\theta=\frac{\omega(z_0)-\zeta}{2+\omega(z_0)-\zeta}$
   if $\omega(z_0)>\zeta$, and
 \begin{equation*}
    |\rho| \le\left\{\begin{array}{cc}  \zeta,
   &\mbox{if}\ \omega(z_0)\le \zeta, \\
    \frac{\omega(z_0)+\zeta}{2+\omega(z_0)-\zeta}, &
    \mbox{if}\ \omega(z_0)> \zeta.
    \end{array}
    \right.
 \end{equation*}
Note that $\omega(z_0)<1$ and $\frac{\omega(z_0)+\zeta}
   {2+\omega(z_0)-\zeta}\le \frac{1+\zeta}{3-\zeta}$
   which is also independent of $K$.
 \end{rem}

The Von Neumann analysis shows that the Robin-Robin
   DD does have a constant rate of convergence,
   independent of the frequency number $k$ or $K$.
But the selection of the two parameters depends on $K$.
The limit case indicates that
    the method deteriorates to, i.e., $\gamma_2=\infty$,
    a Robin-Dirichlet DD method.

\section{Convergence on uniform grids}
In this section, we analyze the Robin-Robin DD method on uniform grids.
In this case, we give explicit eigenvalues of the iterative matrix,
  and show the optimal rate of convergence.

We post a uniform grid of size $h=1/(2n)$ on the domain $\Omega=[0,1]^2$,
   the unit square,  shown in Figure \mref{u-grid}.
Then, we subdivide the domain into two, as shown in Figure \mref{number}.
We give two numberings of nodal values of the $C^0$-$P_1$ 
  finite element functions.
One numbering is on the interface $\Gamma$.
The other  one is within each subdomain, $\Omega_1$ and $\Omega_2$.
 When numbering the
  nodes in $\Omega_2$,  we go from right to left so that the
  nodal index is symmetric to that on the left domain $\Omega_1$.

 \begin{figure}[htb] \setlength{\unitlength}{1.8pt} \begin{center}
    \begin{picture}(160,70)(0,0)
    \def\bg{\begin{picture}(60,60)(0,0)
      \multiput(0,0)(0,60){2}{\line(1,0){60}}
     \multiput(0,0)(30,0){3}{\line(0,1){60}}
      \put(28,61){$\Gamma$}
    \put(1,1){$\Omega_1$}  \put(53,1){$\Omega_2$}
   \multiput(10,10)(0,10){5}{\multiput(0,0)(10,0){5}{\circle*2}}
    \end{picture}}
   \put(0,0){\begin{picture}(60,60)(0,0) \put(0,0){\bg}
   \put(31.5,11.5){$\b x_1$}\put(31.5,21.5){$\b x_2$}
   \put(31.5,31.5){$\b x_3$}\put(31.5,52.5){$\b x_{2n-1}$}
    \end{picture}}
    \put(80,0){\begin{picture}(60,60)(0,0) \put(0,0){\bg}
  \put(8,11.5){1} \put(31.5,11.5){$ $}
   \put(49,11.5){1}
  \put(8,21.5){$2$} \put(31.5,21.5){$ $} \put(49,21.5){$2$}
  \put(-6,66.5){$2n-1$} \put(10,63.5){\vector(0,-1){12}}
  \put(41,61){\vector(-1,-1){10}} \put(63,50){\vector(-1,0){12}}
   \put(65.5,48.5){$2n-1$}
   \put(35.5,64.5){$(2n-1)n$(last)}
    \end{picture}}
 \end{picture}\end{center}
\caption{ Nodal basis numberings, on $\Gamma$ and on
	$\Omega_1\cap\Omega_2$. }
\lab{number}
\end{figure}
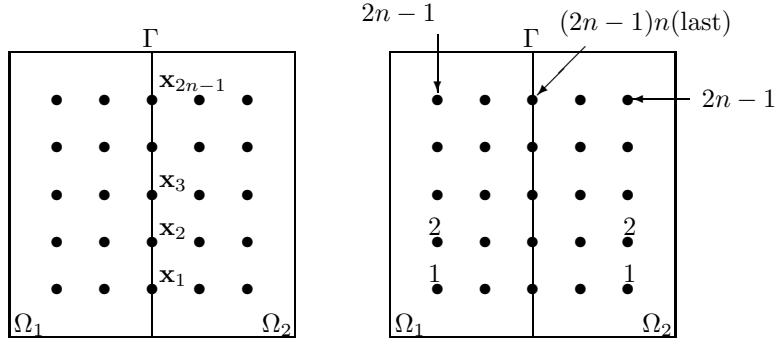
Let $M_{\Gamma}$ and $A_{\Gamma}$ be two tridiagonal ${(2n-1)\times(2n-1)}$
  matrices:
\a{ M_\Gamma = \frac h 6 \p{ 4 & 1 \\ 1 & 4 & \ddots \\
     &\ddots & \ddots  &1 \\ && 1 &4 },
    \quad A_{\Gamma} = \frac{1}{2} \p{ 4 & -1 \\ -1 & 4 & \ddots \\
     &\ddots & \ddots  &-1 \\ && -1 &4 }.
     }
Here $M_\Gamma$ is just the mass matrix of the inner product
   $\langle\cdot,\cdot\rangle$.
Let $R_h$ be the $(2n-1)\times(2n-1)n$ matrix representing
   a restriction operator on $\Gamma$:
 \an{ \lab{R-h}  R_h=(0_{2n-1}, \cdots,0_{2n-1},I_{2n-1}).
 }
The stiffness matrix of the bilinear form $a_1(\cdot,\cdot)$,
  under nodal basis, 
   (and $a_2(\cdot,\cdot)$ too) is
\a{
   A_h=A_0 -R_h^TA_{\Gamma}R_h,
}
where the matrix $A_0$ is the stiffness matrix of size $(2n-1)n$,
   for the Laplace operator on a $(2n)\times(n+1)$ uniform grid with
    zero Dirichlet boundary condition.
$A_0$ is same as the matrix of standard five-point finite difference
matrix, which  has the eigen-decomposition 
   \cite{Demmel1997,Greenbaum1997}:
\e{ \lab{A-h}
 A_0 = (\Phi_n\otimes\Phi_{2n-1})^T
    (\Lambda_n\otimes I_{2n-1} + I_n \otimes \Lambda_{2n-1})
    (\Phi_n\otimes\Phi_{2n-1}),
}
where $\Lambda_m$ denotes an diagonal matrix whose $(i,i)$-th entry is
 \an{\lab{phi-m-i}
      \lambda_i^{(m)}=   4\sin^2\frac{i \pi}{2(m+1)}, }
  and    $\Phi_m$ denotes an orthogonal matrix defined by
 \an{\lab{phi-m}
     \Phi_m = \p{\phi_1^{(m)} & \cdots & \phi_m^{(m)}  },
	\quad\hbox{with } \  \phi_i^{(m)}=\sqrt{\frac{2}{m}}
	\p{ \sin\frac{i\pi}{m+1}\\ \sin\frac{2i\pi}{m+1}\\\vdots
	\\ \sin\frac{mi\pi}{m+1}}. }
 Here in \meqref{A-h},
     a tensor product matrix $C_{mk\times mk}=A_{m\times m}
	\otimes B_{k\times k}$
          is defined with
   the $(i,j)$-th entry
   \a{C_{ij}= A_{i',j'}B_{i'',j''}, \quad \hbox{where }\
	i&= (i'-1) k+i'',\\
	j&= (j'-1) k+j''. }

In Definition \mref{add}, for \meqref{s1}, the error $e_u^m=u-u^m$
   satisfies the  equation:
  \a{ a_1( e_u^m, v) + \gamma_1\langle e_u^m,v\rangle
       & =    \langle e_{g_1}^m,v\rangle
    \quad \forall v\in V_1.
    }
Here $e_{g_1}^m=g_1-g_1^m$ is the error.
In the matrix-vector form,
  \a{ E_u^m = (A_h + \gamma_1R_h^TM_{\Gamma}R_h)^{-1}
      R_h^T M_\Gamma E_{g_1}^m. }
Here $E_u^m$ is the vector representation of $e_u^m$.
Therefore, by \meqref{s3},
   \an{\lab{Eg2n}
     E_{g_2}^m = \left(-I +(\gamma_2+\gamma_1)R_h(A_h
      + \gamma_1 R_h^TM_{\Gamma}R_h
            )^{-1}R_h^TM_\Gamma \right)
            E_{g_1}^m.}
Symmetrically, by \meqref{s4} and \meqref{s5},
    \an{\lab{Eg1nt}
     \tilde E_{g_1}^m = C_{\gamma_2}
            E_{g_2}^m.}
Here, for simplicity,  we denote the error reduction matrix by
	\an{\lab{Cr} C_{\gamma_k}
	= \left(-I +(\gamma_2+\gamma_1)R_h(A_h
      + \gamma_k R_h^TM_{\Gamma}R_h
            )^{-1}R_h^TM_\Gamma \right), \ k=1,2. }
Finally, by \meqref{s5}, one Robin-Robin DD iteration reduces the
   initial error $ E_{g_1}^m$ to
   \an{\lab{Egf} E_{g_1}^{m+1} &= [\theta I + (1-\theta)C_{\gamma_2}
	C_{\gamma_1}] E_{g_1}^m. }
We find the eigenvalue range of this
   error reduction matrix, via common eigenvectors of all matrices.

\begin{lemma}\lab{big-m} The error reduction matrix \meqref{Egf}
  can be diagonalized by $\Phi_{2n-1}$ defined in \meqref{phi-m}.
   That is,
	\an{\lab{rate}
	 \Phi_{2n-1} [\theta I + (1-\theta)C_{\gamma_2}C_{\gamma_1}]
	 \Phi_{2n-1}^T =\operatorname{diag}
	  (\theta + (1-\theta) c_j), }
   where in the $j$-th diagonal element,
  \an{ c_{j}
	=\frac{\gamma_1a_j  -b_j }
              {\gamma_1a_j  +b_j  }
 	\cdot \frac{\gamma_2 a_j  -b_j  }
              {\gamma_2 a_j  +b_j  }.
	\lab{cj}
     }
   Here in \meqref{cj},
   \an{\lab{a-j}
	 a_j &=\lambda_{M_\Gamma, j}\tilde \lambda_j,
	& \lambda_{M_\Gamma, j} &=  h-\frac h6 \lambda^{(2n-1)}_j
      , \\ \lab{b-j}
	 b_j &= 1-\lambda_{A_\Gamma, j}\tilde \lambda_j,
	&
        \lambda_{A_\Gamma, j}&=  1+\frac 12 \lambda_j^{(2n-1)}, }
   where  $\lambda^{(2n-1)}_j$ is defined in \meqref{phi-m-i} and
	\an{ \lab{t-lambda}
   \tilde \lambda_{j}=  {\frac{2}{n+1}}\sum_{i=1}^n\sin^2\frac{in\pi}{n+1}
   (\lambda^{(n)}_i+ \lambda^{(2n-1)}_j)^{-1}.
}
\end{lemma}

\begin{proof}	In \meqref{Eg2n},
    by the Sherman-Morrison-Woodbury formula,
 \a{
    &\quad\  (A_h + \gamma_1 R_h^TM_{\Gamma}R_h)^{-1} \\
    & = (A_0 + R_h^T(-A_\Gamma + \gamma_1 M_{\Gamma}) R_h)^{-1}
   \\ &=A_0^{-1} -A_0^{-1}R_h^T
    (( -A_{\Gamma}+\gamma_1 M_\Gamma)^{-1}+R_hA_0^{-1}R_h^T)^{-1}R_hA_0^{-1}.
 }
 Now letting $B_0=R_hA_0^{-1}R_h^T$, we have
 \a{
  R_h(A_h + \gamma_1 R_h^TM_{\Gamma}R_h)^{-1}R_h^T
  = B_0 - B_0 ( ( -A_{\Gamma}+\gamma_1 M_\Gamma)^{-1}+
     B_0)^{-1}B_0.
 }
 By \meqref{R-h} and \meqref{A-h}, notice
   that $(\Phi_n\otimes \Phi_{2n-1})R_h^T=\phi_n^{(n)} \otimes \Phi_{2n-1}$,
    we can compute $B_0$:
 \a{
   B_0 %&= R_h (\Psi_0^T D_0^{-1}  \Psi_0) R_h^T  \\
    &= (\phi_n^{(n)} \otimes \Phi_{2n-1})^T  (\Lambda_n\otimes I_{2n-1} + I_n \otimes \Lambda_{2n-1})% D_0^{-1}
	  (\phi_n^{(n)} \otimes \Phi_{2n-1}) \\
    &= \sum_{i=1}^n (\phi_{n,i}^{(n)} )^2
	  \Phi_{2n-1} ^T (\lambda^{(2n-1)}_i I_{2n-1} +\Lambda_{2n-1})^{-1}
	  \Phi_{2n-1}  \\
  &=  \Phi_{2n-1} ^T \left( \sum_{i=1}^n (\phi_{n,i}^{(n)} )^2
	  (\lambda^{(2n-1)}_i
	 I_{2n-1} +\Lambda_{2n-1})^{-1} \right)  \Phi_{2n-1}  \\
     &=  \Phi_{2n-1} ^T \tilde \Lambda_0  \Phi_{2n-1},
 } where $\phi_{n,i}^{(n)} $ is the $i$-th entry of vector
	$\phi_n^{(n)}$ defined in \meqref{phi-m},
	and  $\tilde{\Lambda}_0$ is a diagonal matrix,
    whose $(j,j)$-th entry is
  define in \meqref{t-lambda}.
The matrices on $\Gamma$ are diagonalized as
   $ M_\Gamma   =\Phi_{2n-1}^T \operatorname{diag}(\lambda_{M_\Gamma,j})
	       \Phi_{2n-1}$ and
      $ A_\Gamma  =\Phi_{2n-1}^T \operatorname{diag}(\lambda_{A_\Gamma,j})
	       \Phi_{2n-1}$, where $\lambda_{M_\Gamma,j}$
	and $\lambda_{A_\Gamma,j}$ are defined in \meqref{a-j} and
	\meqref{b-j}, respectively.
Thus combining last two equalities, we get
 \a{ &\quad \  R_h(A_h + \gamma_1 R_h^TM_{\Gamma}R_h)^{-1}R_h^T \\
  & = B_0 - B_0 (  \Phi_{2n-1}^T( -\operatorname{diag}(\lambda_{A_\Gamma,j})+
      \gamma_1 \operatorname{diag}(\lambda_{M_\Gamma,j}) )^{-1}\Phi_{2n-1} +
     B_0)^{-1}B_0 \\
  & = \Phi_{2n-1}^T \left[ \tilde{\Lambda}_0 -  \tilde{\Lambda}_0^2
	\{ (-\operatorname{diag}(\lambda_{A_\Gamma,j})+
      \gamma_1 \operatorname{diag}(\lambda_{M_\Gamma,j}) ) ^{-1}
	+ \tilde{\Lambda}_0 \}^{-1} \right] \Phi_{2n-1} \\
	 &= \Phi_{2n-1}^T \left[ \tilde{\Lambda}_0^{-1}
             -\operatorname{diag}(\lambda_{A_\Gamma,j})+
      \gamma_1 \operatorname{diag}(\lambda_{M_\Gamma,j})  \right]^{-1}
           \Phi_{2n-1}.
 }
  By \meqref{Cr},
   \a{ C_{\gamma_1} &= \Phi_{2n-1}^T (-I+(\gamma_2+\gamma_1)
       \left[ (\tilde{\Lambda}_0^{-1}
             -\operatorname{diag}(\lambda_{A_\Gamma,j}))
        \operatorname{diag}(\lambda_{M_\Gamma,j}^{-1} )+
      \gamma_1 I )  \right]^{-1}
           )  \Phi_{2n-1}
	\\ &= \Phi_{2n-1}^T
      \operatorname{diag}\left(
	\frac{-1+ \gamma_2\lambda_{M_\Gamma,j} \tilde\lambda_j+
	   \lambda_{A_\Gamma,j} \tilde\lambda_j }
	{1+ \gamma_1  \lambda_{M_\Gamma,j}\tilde\lambda_j -
             \lambda_{A_\Gamma,j} \tilde\lambda_j } \right)  \Phi_{2n-1}
	\\ &= \Phi_{2n-1}^T
      \operatorname{diag}\left(
	\frac{ \gamma_2a_j-b_j }
	{ \gamma_1 a_j + b_j  } \right)  \Phi_{2n-1}.
	}
  In the same fashion, it follows that
  \a{ C_{\gamma_2} &=   \Phi_{2n-1}^T
      \operatorname{diag}\left(
	\frac{ \gamma_1a_j-b_j }
	{ \gamma_2 a_j + b_j  } \right)  \Phi_{2n-1}.
	}
 Thus \meqref{rate} follows.
\end{proof}

In next lemma,  we estimate the eigenvalue $c_j$ in
	the reduction matrix,  \meqref{cj}.

% \theta_j = (2(1+h)-(1-h)\cos\frac{j\pi}{2n}),$
\begin{lemma}\label{ttld01} $(3a_j-b_j)$ is monotonically decreasing,
  i.e.,  $j=1,\dots,
    2n-2$,
  \e{ \lab{3-a-b} 3a_j-b_j \ge 3a_{j+1}-b_{j+1}.}
\end{lemma}
\begin{proof}
We rewrite the $\tilde\lambda_{ j}$ (in $a_j$ and $b_j$)
   in a symmetric form so that each $i$-term is a decreasing function
   of $j$ (the original term is not.)
  \a{
     \tilde\lambda_{j}  &=  \frac{2}{n+1} \frac 12 \sum_{i=1}^n
     \frac{ \sin^2 ( i \pi/(n+1)) }
            { 4\sin^2({j\pi}/(4n))+4\sin^2({i\pi}/({2n+2}))}
     \\
       \nonumber &\qquad +  \frac{ \sin^2 ((n+1- i) \pi/(n+1)) }
     { 4\sin^2({j\pi}/(4n))+4\sin^2({(n+1-i)\pi}/({2n+2}))}
         \\
        &=  \frac{1}{n+1} \sum_{i=1}^n
     \frac{  \sin^2 ( i \pi/(n+1))
                    (2 \sin^2({j\pi}/(4n))+1) }
            {4\sin^4({j\pi}/(4n))+4\sin^2({j\pi}/(4n))
    +\sin^2( i \pi/(n+1))}   .
     \nonumber
   }

To shorten expression,  we introduce two more notations
	\an{\lab{xi-j}  \xi_j& = \sin^2 \frac {j\pi}{4n}, \\
	\lab{theta-j}
	      \theta_j & =(1+2 \xi_j )
	  (1+3h+2\xi_j -2h\xi_j ) .}
By \meqref{a-j} and \meqref{b-j}, we have
 \an{\lab{t-a}   3a_j-b_j+1  =   \frac{1}{n+1} \sum_{i=1}^n
     \frac{  \sin^2 ( i \pi/(n+1)) \theta_j  }
            {4\xi_j^2 +4\xi_j  +\sin^2( i \pi/(n+1))}.}
We show that each term is a decreasing function of $\xi_j$.
That is, each term \a{ f_i(\xi) =  \frac{
                    (2 \xi +1)
        (2(1+h)-(1-h)(1-2\xi))  }
            {4\xi ^2 +4\xi  +\sin^2( i \pi/(n+1))} }
  is a decrease function of $\xi$, for $\xi\in(0,1)$.
By the quotient rule,
\a{ f_i'(\xi) &= \frac{
    ( 4(1+h) +8(1-h)\xi )(4\xi^2+4\xi+\sin^2( i \pi/(n+1))) }
     {( 4\xi ^2 +4\xi  +\sin^2( i \pi/(n+1)))^2 }\\
    &\quad - \frac{ ((1+3h)+4(1+h)\xi+4(1-h)\xi^2) (8\xi+4)}
     {( 4\xi ^2 +4\xi  +\sin^2( i \pi/(n+1)))^2 }. }
  % 1: 4 (1+h)s  -4 -12 h
   % S: 16(1+h) + 8(1-h)s - 8(1+3h)-16(1+h) =
   %S^2: 16(1+h) +32(1-h) -32(1+h) -16(1-h)
    %S^3 : 32(1-h) - 32(1-h) = 0
The combined numerator is
  \a{ &\quad -\left(4(1+h)\cos^2\frac{i\pi}{n+1} +8h\right)
    -\left( 8 (1-h) \cos^2\frac{i\pi}{n+1} + 16 h\right) \xi
    - (32h) \xi^2 <0.
     }
As each term $f_i(\xi_j)$ is desecrating with respect to $j$,
   the sum is a desecrating function of $j$. We prove the lemma.
 \end{proof}

We will find a bound for
   the biggest term $(3a_1-b_1)$, among all $(3a_j-b_j)$,
	in order to bound the $c_j$ in \meqref{cj}.
One can prove that, for all $n\ge 1$,
	\an{ 3a_1-b_1   < -7 h^2/16, \lab{-h-2} }
    where $a_1$ and $b_1$ are defined in \meqref{a-j} and
	\meqref{b-j}, respectively, and $h=1/(2n)$.
But our proof for \meqref{-h-2} is lengthy and tedious.
In this paper, we prove a worse bound only, in the next lemma.

%\frac{1}{8}\le \tilde \lambda_{j}\le \tilde \lambda_{1}
% \le 1-\frac{3/2+ \beta}{n+1}, \quad 1\le j\le (2n-1),

\begin{lemma}\label{lambda01} If $n\ge 11$, then (cf. \meqref{-h-2})
 \an{ 3a_1-b_1 <-0.049 h  < -7 h^2/16, \lab{-h-1} }
    where $a_1$ and $b_1$ are defined in \meqref{a-j} and
	\meqref{b-j}, respectively.
\end{lemma}

\begin{proof}  By \meqref{t-a}, with the notations defined
     in \meqref{a-j}, \meqref{b-j} and \meqref{t-lambda},
  \an{\lab{a1-b1} 3a_1-b_1+ 1 &=  (1+3h+\frac{1-h}2\lambda_1^{(2n-1)})
	\tilde \lambda_{1}. }
 We estimate an upper bound for
\an{ \nonumber
	 \tilde \lambda_{1} &=
   {\frac{2}{n+1}}\sum_{i=1}^n \cos^2\frac{i\pi}{2(n+1)}
     -  {\frac{2}{n+1}}\sum_{i=1}^n \frac{\cos^2\frac{i\pi}{2(n+1)}
     \sin^2\frac{\pi}{4n}}
    { \sin^{2}\frac{i\pi}{2(n+1)}+ \sin^2\frac{\pi}{4n}}\\
 \nonumber
	 & =  {\frac{n}{n+1}}
     -  {\frac{2\sin^2\frac{\pi}{4n}}{n+1}}
  \sum_{i=1}^n \left(\frac{ 1+\sin^2\frac{\pi}{4n}}
    { \sin^{2}\frac{i\pi}{2(n+1)}+\sin^2\frac{\pi}{4n}}-1\right)\\
   &=1-\frac{1- 2n\sin^2\frac{\pi}{4n}}{n+1}
   -\frac{2(1+\sin^2\frac{\pi}{4n})}{n+1}\sum_{i=1}^n
         \frac{\sin^2\frac{\pi}{4n} }
    {  {\sin^{2}\frac{i\pi}{2(n+1)}}+{\sin^2\frac{\pi}{4n}} }.
  \lab{l-1}  }
As $(\sin x/x)$ is a decreasing function of $x$ on $(0,\pi/2)$, we have
  \a{ \sum_{i=1}^n
         \frac{\sin^2\frac{\pi}{4n} }
    {  {\sin^{2}\frac{i\pi}{2(n+1)}}+{\sin^2\frac{\pi}{4n}} }
     & >
 \sum_{i=1}^n \frac{1  }
    {\left( \left. \frac{i\pi}{2(n+1)} \right/ \frac{\pi}{4n} \right)^2+1}
     >  \sum_{i=1}^{n}\frac{1}{1+4i^2}\\
	&  \ge  \sum_{i=1}^{11}\frac{1}{1+4i^2}  >0.33462.
 }
Substituting the estimate into the expression of $\tilde{\lambda}_1$,
\a{
 \tilde{\lambda}_1< 1-\frac{1+2(1+2\sin^2\frac{\pi}{4n})\cdot0.33462
	-  2n\sin^2\frac{\pi}{4n}}{n+1}
  < 1-\frac{1.55726}{n+1}.
}
By \meqref{a1-b1}, if $n\ge 11$,
  \a{ 3a_1-b_1+ 1 &< (1+ \frac 3{2n} + \frac{\pi^2}{8n^2})
	(1-\frac{1.55726}{n+1} )\\
	&< 1 -0.049 h \le 1-0.98 h^2< 1 -7 h^2/16. }
We proved the lemma.
\end{proof}

With the explicit eigenvalues of the reduction matrix
	and their bounds,
  we can easily choose a set of parameters $\gamma_1$, $\gamma_2$
	and $\theta$, to get a constant rate of reduction, independent of
	mesh size $h$.
\begin{theorem}\lab{main}
    Let $\gamma_1=1$, $\gamma_2=64 h^{-1}$
    and $\theta=3/7$ in Definition \mref{add}.
   The error reduction factor (for the $P_1$ finite element on uniform grids 
	shown in Figure \ref{u-grid})  
 is bounded by $1/7$, independent of the
    grid size $h$, %that is, cf. \meqref{Eg},
   \a{  \|e^{m+1}_{g_1}\|_{L^2(\Gamma)}
    \le \frac 17 \|e^{m}_{g_1}\|_{L^2(\Gamma)}.}
\end{theorem}
\begin{proof} We will apply Lemma \mref{lem:fx}.
By \meqref{a-j} and \meqref{phi-m-i}, $a_j>0$.
    By \meqref{3-a-b},  \meqref{-h-2} and \meqref{b-j},
     \an{ \nonumber 3a_j-b_j &  \le 3a_1-b_1 \le - 7h^2/16, \\
	  \lab{7}    b_j &  \ge  3a_j + 7h^2/16 >0. }
By \meqref{cj},
  \a{ c_j = \frac{ 1   -b_j/a_j }
              {1  +b_j/a_j   }
 	\cdot \frac{64h^{-1}    -b_j/a_j  }
              {64h^{-1}  +b_j/a_j   }. }
We let $z=b_j/a_j>0$ in Lemma \mref{lem:fx}.
The critical point is (may be outside the $b_j/a_j$ range)
  \a{ z_0 = \sqrt{ \gamma_1\gamma_2 } = 8h^{-1/2}. }
We find the two end points of possible  $z$.
First, by  \meqref{t-lambda},
	\a{  \tilde \lambda_j &\ge \frac 2{n+1}
             \sum_{i=1}^n \frac 18 \sin^2\frac{i n \pi}{n+1}
	=\frac n{8(n+1)} > \frac 18. }
Thus, by \meqref{a-j}, \meqref{b-j} and \meqref{phi-m-i},
   \a{ a_j & \le (h-\frac h6 \cdot 0) \cdot 1=h, \\
	a_j & \ge (h-\frac h6 \cdot 4) \cdot \frac 18 =\frac h{12}, \\
       b_j &\le  1-(1+\frac 12  \cdot 0 )\cdot \frac 18 =\frac 78.}
In the first inequality, we used \meqref{l-1} that $\tilde \lambda_j<1$.
We find one end point for $z$:
	\a{ \frac{ b_j }{a_j} & \le \frac{ 7/8 }{ h/12 } = \frac{21}{2h}
	\equiv z_r. }
For the other end point, by  \meqref{7},
     \a{ \frac{ b_j }{a_j}
	   &  \ge 3 + \frac {7h^2/16}{a_j } \ge
		3 + \frac {7h^2/16}{h} = 3 + \frac{7h}{16}\equiv z_l. }
By Lemma \mref{lem:fx}, the range of $c_j$ is between its values
	at $z=z_l, z_0, z_r$.
We note that $z_l<z_0<z_r$ here.
At each point,  we need to apply \mref{lem:fx} again for $h$ varying.
But we can find some rough (but good enough) bounds at each point, directly.
  \a{ \hbox{At $z=z_r$:}&  &  -0.718...=
	-\frac{107}{149}
       &< c_j
	%\frac{ 1   - 21/(2h) } {1  + 21/(2h)   }
 	%\cdot \frac{64h^{-1}    - 21/(2h)  } {64h^{-1}  + 21/(2h)  }
        \le -\frac{1070}{1639}=  -0.65...  \\
     \hbox{At $z=z_l$:} &  &
	 -0.50098...=-\frac{2184975 }{4361329}&\le c_j
	%\frac{ 1   -3 - 7h/16 } {1  + 3 + 7h/16   }
 	%\cdot \frac{64h^{-1} -3 - 7h/16  } {64h^{-1} + 3 + 7h/16  }
          <-\frac{1}{2}=-0.5. \\
    \hbox{At $z=z_0$:} &  &
	 -1&< c_j % \frac{ 1   -8h^{-1/2} } {1  +8h^{-1/2}  }
 	     %\cdot \frac{64h^{-1} - 8h^{-1/2}  } {64h^{-1} +  8h^{-1/2} }
           \le \frac{32-129\sqrt 2}{32-129\sqrt 2}=-0.7015...
}
\comment{z=0.04:0.02:3; f=(z-0.1).*(z-20)./(z+0.1)./(z+20);
 plot(z,f);
h=0.01:0.01:0.5;  h=1/4; z=3+7*h/16; (1-z).*(64./h-z)*64
w=(1-z).*(64./h-z)./(1+z)./(64./h+z);
h=0.01:0.01:0.5;   z=8*h.^(-1/2);
w=(1-z).*(64./h-z)./(1+z)./(64./h+z);
plot(h,w)

h:=1/2;
a:=expand(((1-8/sqrt(h))*(64/h-8/sqrt(h))))/8;
b:=expand(((1+8/sqrt(h))*(64/h+8/sqrt(h))))/8;
evalf(a/b);
}

Hence the value of $c_j$ is always strictly between $-1$ and $-1/2$.
When $\theta=3/7$, we get,
\an{\lab{t1}    \theta +(1-\theta)  c_j
     &> \frac 37 +\frac 47 (-1) = -\frac 17, \\
    \lab{t2}    \theta +(1-\theta)  c_j
     & < \frac 37 +\frac 47 (-\frac 12) = \frac 17.
 }
This gives the error reduction factor.\end{proof}

By \meqref{t1} and \meqref{t2},  we can get the following result for
   a general relaxation parameter $\theta$.

\begin{corollary} \lab{co}
    Let $\gamma_1=1$ and $\gamma_2=64 h^{-1}$ in Definition \mref{add}.
   The error reduction factor $\rho$ for
   the $P_1$ finite element on uniform grids 
	(shown in Figure \ref{u-grid}) is
   \a{  \rho=\begin{cases}
      1-2\theta , & 0\le\theta\le 3/7,\\
	(3\theta -1)/2,  & 3/7 < \theta\le 1. \end{cases}  }
  That is,  $\|e^{m+1}_{g_1}\|_{L^2(\Gamma)}
	 \le  \rho\|e^{m}_{g_1}\|_{L^2(\Gamma)}$.
\end{corollary}

%---------------------------------------------------------------
\section{Convergence on general grids}

In this section, we consider the convergence behavior  of the Robin-Robin
  DD method on general quasi-uniform meshes.
By the algorithm in Definition \mref{add}, for $i=1,2$,
\a{ a_i(e_i^m , v) +\gamma_i \langle e_i^m, v\rangle =\langle
   \varepsilon_i^m, v\rangle, \quad \forall v \in V_i}
where the errors are defined by 
\a{ \varepsilon_i^m=g_i-g_i^m, 
  \quad e_1^m =u-u^m, \ \hbox{ \ and \ } \ e_2^m =w-w^m. }

Let $S_1$ and $ S_2$ be the standard Dirichlet-to-Neumann operators,
    cf. \cite{QV,XQ}.
The error $\varepsilon_i^n$ ($i=1,2$),
   restricted to the
   interface $\Gamma$, satisfies the relation
     \e{ \lab{ve}\varepsilon_i^m=(\gamma_i + S_i) e_i^m |_\Gamma.}
Using the first interface update \meqref{s2}, we have
     \e{\lab{ve1} \varepsilon_2^m=-\varepsilon_1^m+
       (\gamma_1+\gamma_2)e_1^m|_\Gamma.}
For the second one, by \meqref{s4} and \meqref{s5},
   \a{
   \varepsilon_1^{m+1}
   & =\theta \varepsilon_1^m +(1-\theta)
       [ -\varepsilon_2^m+(\gamma_1+\gamma_2) e_2^m|_\Gamma] \\
      & =\theta \varepsilon_1^m+(1-\theta)
     [ -(\gamma_2+ S_2)e_2^m|_\Gamma+(\gamma_1+\gamma_2) e_2^m|_\Gamma]\\
      & = \theta \varepsilon_1^m+(1-\theta)(\gamma_1 -S_2) e_2^m|_\Gamma.
      }
By \meqref{ve}, \meqref{ve1}, we have
   \a{
    \varepsilon_1^{m+1} & =\theta \varepsilon_1^m
 + (1-\theta)(\gamma_1 -S_2)(\gamma_2+ S_2)^{-1}  \varepsilon_2^m\\
     & =\theta \varepsilon_1^m + (1-\theta)(\gamma_1 -S_2)
    (\gamma_2+ S_2)^{-1}( \gamma_2-S_1) e_1^m|_\Gamma \\
      & =[\theta +(1-\theta)(\gamma_1 -S_2)(\gamma_2+ S_2)^{-1}
     ( \gamma_2-S_1)(\gamma_1+S_1)^{-1}]\varepsilon_1^m.
   }
Let us represent the iteration by
  \a{ \varepsilon_1^{m+1}=R \, \varepsilon_1^m, }
where \an{ \lab{R} R&= \theta  - (1-\theta)T,
   \\ \lab{T} T& =(  S_2- \gamma_1)(\gamma_2+ S_2)^{-1}
    ( \gamma_2-S_1)(\gamma_1+S_1)^{-1}. }
 Next, we give an convergence analysis for this DD operator $R$.

 \subsection{Symmetric case: $S_1=S_2(:=S)$}

 Let $z$ be an eigenvector of the symmetric operator $S$ (cf. \cite{QV,XQ})
    corresponding to the eigenvalue $\lambda_s$.
 By \meqref{T}, $z$ is also an eigenvector of the symmetric operator $T$.
   \a{ [\theta + (1-\theta)T] z &=
  [\theta + (1-\theta) \frac{(\gamma_1-\lambda_s)(\gamma_2-\lambda_s)}
     {(\gamma_1+\lambda_s)(\gamma_2+\lambda_s)}]z
    \\& =[\theta-(1-\theta)\omega(\lambda_s)]z.}
 It is known \cite{XQ} that
$$ \lambda_s \in [c_0, C_0 h^{-1}]. $$
Now if we choose
  \an{\lab{p2} 0<\gamma_1\le c_0, \ \hbox{ and } \
          \gamma_2\ge C_0h^{-1}, }
 by Lemma \mref{lem:fx},
 we get
 \a{
   0\le \omega(\lambda_s)\le \frac{(\eta-1)^2}{(\eta+1)^2},
     \quad \eta=\sqrt\frac{\gamma_2}{\gamma_1}.
 }
Then we bound the spectrum of the symmetric operator $R$, $\sigma(R)$, as
 \a{ \sigma(R)  \subset
   [\theta-(1-\theta)\frac{(\eta-1)^2}{(\eta+1)^2},\theta]
	\subset [-\frac 13, \frac 13],
 }
%   $$ [\theta, \theta +(1-\theta) \frac{1-c_0}{1+c_0}],$$
when choosing the parameter $\theta=1/3$, cf. Remark \mref{r1}.
That is, the  convergence rate is bounded by $1/3$,
  independent of the mesh size $h$, when choosing parameters by
    \meqref{p2}.

\subsection{Nonsymmetric case: $S_1\approx S_2$}
 In this case, there exist two positive
constant $0<s\le 1$ and $t\ge 1$, independent of the grid size $h$,
   such that for all
 $v\in {V_i}|_{\Gamma}$  (cf. \cite{QV} for details):
\e{\tag{A1}\lab{A-1}
   s(S_1v,v) \le (S_2v,v)\le t(S_1v,v).
}

$S_i$($i=1,2$) are symmetric and positive definite(SPD).
Let $\underline{\lambda}_i$ be the minimum eigenvalue,
 and $\overline{\lambda}_i$   the maximum eigenvalue of $S_i$.
 In this subsection, we assume that the parameters are chosen to satisfy
\e{\tag{A2}\lab{A-2}
   0<\gamma_1 \le \min\{ \underline{\lambda}_1, \ \underline{\lambda}_2\},
     \quad\mbox{and}\quad
    \gamma_2\ge 3 \max\{ \overline{\lambda}_1, \ \overline{\lambda}_2\}.
 }
 The parameter selection is similar to that in the symmetric case,
   \meqref{p2}.

\begin{lemma}\lab{rmkc1}   The condition \meqref{A-1} has another version
  \an{\lab{e-A-1}
    \frac{1}{t}(S_1^{-1}v,v)\le (S_2^{-1}v, v)\le \frac{1}{s}(S_1^{-1}v,v).
  }
\end{lemma}

\begin{proof} Replacing $v$ by $ S_1^{-\half}v$ in \meqref{A-1},
\a{
  s(v,v)\le (S_1^{-\half} S_2S_1^{-\half}v, v)\le t(v,v).
}
This inequality implies that the spectrum of the SPD operator
   $S_1^{-\half}S_2S_1^{-\half}$ is within $[s,t]$.
So the spectrum of its inverse, $S_1^{ \half}S_2^{-1}S_1^{ \half} $
   is inside $[t^{-1},s^{-1}]$, i.e.,
\a{
   \frac{1}{t}(v,v)\le (S_1^{\half}S_2^{-1}S_1^{\half}v,v)
   \le \frac{1}{s}(v,v).
}
\meqref{e-A-1} follows after replacing $v$ by $S^{-\frac12} v$.
\end{proof}

\comment{
n=100; x=1:n; s=1/2; t=2; s1=x; r1=1/4; r2=4*n; th=1/3;
TpTP=(r2-s1)./(r1+s1).*(r1-s*s1)./(r2+s*s1);
TTP=((r2-s1)./(r1+s1).*(r1-t*s1)./(r2+t*s1)).^2;
for i=1:10, th=i/11
 % plot(th^2+th*(1-th)*2*TpTP+(1-th)^2*TTP)
plot(x,2*TpTP,'-',x,TTP)
pause;
end

}

To find the spectrum of DD operator $T$ in \meqref{T}, we introduce
 a symmetric  operator
\e{\lab{t-T}
  \tilde{T} =(\gamma_1+S_1)^{-\half} (\gamma_2-S_1)^{\half}
  (S_2-\gamma_1)(\gamma_2+S_2)^{-1}(\gamma_2-S_1)^{\half}
   (\gamma_1+S_1)^{-\half}.
}
This operator is similar to the nonsymmetric operator
     $T$, defined in \meqref{T}.

\begin{lemma} \lab{t-T-s} If Assumption  \meqref{A-2}
    is satisfied,  then $\tilde{T}$ is SPD.
\end{lemma}

\begin{proof} $\tilde{T}$ is symmetric because
 \a{ \tilde{T}^T &= (\gamma_1+S_1^T)^{-\half} (\gamma_2-S_1^T)^{\half}
  (\gamma_2+S_2^T)^{-1}(S_2^T-\gamma_1)
	 (\gamma_2-S_1^T)^{\half}
   (\gamma_1+S_1^T)^{-\half}\\ &=\tilde T. }
 Notice that $(S_2-\gamma_1)(\gamma_2+S_2)^{-1}
    =I -(\gamma_1+\gamma_2)(\gamma_2+S_2)^{-1}$.
  Its minimum eigenvalue is
	\a{ 1-(\gamma_1+\gamma_2)(\gamma_2+\underline{\lambda}_2)^{-1}
	=\frac{\underline{\lambda}_2-\gamma_1}
       {\gamma_2+\underline{\lambda}_2},}
 which is positive by Assumption \meqref{A-2}.
  Similarly, the  minimum eigenvalue of
  $(\gamma_2-S_1)(\gamma_1+S_1)^{-1}$
	is $(\gamma_2-\overline{\lambda}_1)/
	(\gamma_1+\overline{\lambda}_1)$,
 which is also positive by Assumption  \meqref{A-2}.
Now for any $v\in V_i|_{\Gamma}$, we have,
   denoting $\tilde v=(\gamma_2-S_1)^{\half}(\gamma_1+S_1)^{-\half}v$,
\a{
   (\tilde{T}v,v)&=   ((S_2-\gamma_1)(\gamma_2+S_2)^{-1}
	\tilde v , \tilde v)\\
   &\ge
    \frac{\underline{\lambda}_2-\gamma_1}{\gamma_2+\underline{\lambda}_2}
  (\tilde v , \tilde v)\\
    &=   \frac{\underline{\lambda}_2-\gamma_1}
  {\gamma_2+\underline{\lambda}_2}((\gamma_2-S_1)(\gamma_1+S_1)^{-1}v,v)\\
    &= \frac{\underline{\lambda}_2-\gamma_1}
	{\gamma_2+\underline{\lambda}_2} \cdot
	\frac{\gamma_2-\overline{\lambda}_1}{\gamma_1+\overline{\lambda}_1}
	(v,v).
   }
It means that the minimum eigenvalue of $\tilde{T}$ is greater than
  $\displaystyle \frac{\underline{\lambda}_2-\gamma_1}
	{\gamma_2+\underline{\lambda}_2}
	\frac{\gamma_2-\overline{\lambda}_1}
	{\gamma_1+\overline{\lambda}_1}>0$. That is to say,
  the symmetric operator $\tilde{T}$ is also  positive definite.
 \end{proof}

We find an upper bound of
     the spectrum of SPD operator $\displaystyle \tilde{T}$ next.
To this end,  we rewrite $\tilde{T}$ as
\e{ \lab{new-t-T}
   \tilde{T} = \tilde{T}_2^T \tilde{T}_1\tilde{T}_2,
}
where
 \an{ \lab{t-T1}
   \tilde{T_1} &=(S_2-\gamma_1)(\gamma_2+ S_2)^{-1}
  ( \gamma_2-S_2)(\gamma_1+S_2)^{-1},\\
   \lab{t-T2}
 \tilde{T}_2 &=(\gamma_1+S_2)^{\half}(\gamma_2-S_2)^{-\half}
   (\gamma_1+S_1)^{-\half}(\gamma_2-S_1)^{\half}.
  }

\begin{lemma}\label{lemct2}
	If  \meqref{A-1}  and  \meqref{A-2} hold,
    then, for the $t$ defined in \meqref{A-1},
\an{\lab{b1}
    ((\gamma_2-S_2)^{-1}(\gamma_1+S_2) v,v)
    &\le ( 2t -1 ) \, ((\gamma_2-S_1)^{-1}(\gamma_1+S_1) v,v).
 }
\end{lemma}

\begin{proof} By \meqref{A-1} and \meqref{e-A-1},
\a{
   ((\gamma_2-S_1)(\gamma_1+S_1)^{-1}v,v)&=
   ((\gamma_1+\gamma_2)(\gamma_1+S_1)^{-1}v,v)-(v,v)\\
 &\le  t\, ((\gamma_1+\gamma_2)(\gamma_1+S_2)^{-1}v,v)-(v,v) \\
  &= t\, ((\gamma_2-S_2)(\gamma_1+S_2)^{-1}v,v) +(t-1)(v,v).
 }
We bound the second term next. By the assumption  \meqref{A-2},
 \a{ ((\gamma_2-S_2)(\gamma_1+S_2)^{-1}v,v)
	&\ge \frac{\gamma_2- \overline{\lambda}_2 }
          {\gamma_1+\overline{\lambda}_2} (v,v)
	 \ge \frac{2\overline{\lambda}_2 }
          {2\overline{\lambda}_2} (v,v) = (v,v) .
 }
Combining above two inequalities,
\an{\lab{A4}
    ((\gamma_2-S_1)(\gamma_1+S_1)^{-1} v,v)
    &\le ( 2t -1 ) \, ((\gamma_2-S_2)(\gamma_1+S_2)^{-1} v,v).
 }
Applying Lemma \mref{rmkc1}, replacing $S_2$ there
   by $(\gamma_2-S_1)(\gamma_1+S_1)^{-1}$ and
   $S_1$ by $(\gamma_2-S_2)(\gamma_1+S_2)^{-1}$, by \meqref{A4},
  \a{ \frac1{2t-1} ((\gamma_2-S_2)^{-1}(\gamma_1+S_2) v,v)
    &\le   ((\gamma_2-S_1)^{-1}(\gamma_1+S_1) v,v). }
 \meqref{b1} is proved.
\end{proof}

\begin{lemma} \label{lemupt} If assumptions \meqref{A-1}  and  \meqref{A-2}
    hold,
then the spectrum of the SPD operator $\tilde{T}$ is bounded by
\e{\lab{sigma-T}   \sigma(\tilde{T}) \subset (0, 2t-1].}
\end{lemma}

\begin{proof}  $\tilde T_1$ is SPD, cf. \meqref{t-T1}.
     The eigenvalues of $\tilde T_1$ are
  \def\lwo{\lambda_{2,j}}
\a{ \tilde \lambda_{j} = \frac{\lwo-\gamma_1}{\gamma_2+\lwo}
	\frac{\gamma_2-\lwo}{\gamma_1+\lwo} , }
 where $\{\lwo\}$ are all eigenvalues of $S_2$.
 By \meqref{A-2},
  \a{  \tilde \lambda_{j} &> \frac{\lwo-\underline{\lambda}_2 }{\gamma_2+\lwo}
	\frac{4 \overline{\lambda}_2-\lwo}{\gamma_1+\lwo} \ge 0, \\
     \tilde \lambda_{j} &< \frac{\lwo }{\gamma_1+\lwo}
	\frac{\overline{\lambda}_2}{\gamma_2+\lwo} <1.
  }
Then, by \meqref{t-T2}, \meqref{new-t-T} and \meqref{b1},
\a{
 0 < (\tilde{T}v,v) & < (\tilde{T}_2v,\tilde{T}_2v)
    =((\gamma_1+S_2)(\gamma_2-S_2)^{-1}
     \tilde v,\tilde v)\\
    &\le (2t-1) ( (\gamma_1+S_1) (\gamma_2-S_1)^{-1} \tilde v,\tilde v)
     = (2t-1) (v,v)	    ,
}where \a{ \tilde v = (\gamma_1+S_1)^{-\frac 12}(\gamma_2-S_1)^{\frac12}v. }
 The proof is completed.
\end{proof}

\begin{theorem} If the assumptions \meqref{A-1} and  \meqref{A-2} hold,
   then the spectrum of DD reduction operator  $R$, defined in \meqref{R},
    is bounded, independent of the grid size $h$:
    \an{\lab{sigma-R}
   \sigma(R) \subset [-\frac {2t-1}{2t+1}, \frac {2t-1}{2t+1}], }
  when $\theta$ is selected by
   \an{\lab{t-s}  \theta= \frac{2t-1}{2t+1}. }
\end{theorem}
\begin{proof} By \meqref{sigma-T} and \meqref{R},
    \a{
    \sigma(R) \subset  [\theta-(1-\theta)(2t-1), \theta],
    }
 where $t$ is defined in \meqref{A-1}, independent of $h$.
Similar to the idea in Lemma \ref{lem:theta},
   \meqref{sigma-R} follows after 
  we choose the optimal $\theta$ by \meqref{t-s}. 
\end{proof}

%---------------------------------------------------------------------------------
\section{Numerical test}

\def\tableOne{5.1}
\def\tableTwo{5.2}
\def\tableTwo{5.3}

For numerical test,  we solve the Poison equation \meqref{laplace}
    on the unit square $[0,1]$.
The exact solution is chosen
  \a{ u(x,y) = 2^6 (x^3-x^4)(y-y^2). }
We choose $x=1/2$ as the domain decomposition interface.
We use $P_1$ conforming finite element on uniform
    criss grids, shown in Figure \mref{u-grid}.

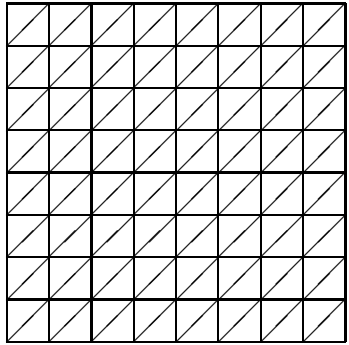
\begin{figure}[htb] \begin{center}\setlength\unitlength{1.6pt}
    \begin{picture}(80,80)(0,0)
  \def\bl{\setlength\unitlength{0.2pt}\begin{picture}(80,80)(0,0)
    \put(0,0){\line(1,0){80}}\put(0,0){\line(0,1){80}}
      \put(80,80){\line(0,-1){80}} \put(80,80){\line(-1,0){80}}
      \put(0,0){\line(1,1){80}} \end{picture}}
   \multiput(0,0)(10,0){8}{\multiput(0,0)(0,10){8}{\bl}}
     \end{picture}\end{center}
\caption{A uniform criss grid of size $h=1/8$.}
\lab{u-grid}
 \end{figure}

First, we do the Robin-Robin iteration (Definition \mref{add})
   for problems with different grid size.
The parameters used
    are $\gamma_1=1$, $\gamma_2=64/h$ and $\theta=3/7$.
The iteration stops when $|g^{m+1}_1 - g^{m}_1|_{l^\infty} < 10^{-11}$.
The number of iteration, the error and the order of convergence
   for the finite element solution are listed in Table \mref{t-1}.
We note that there is a superconvergence for the finite element
   solution in semi-$H^1$ norm.

\begin{table}[htb]
  \caption{ The errors  and the iteration numbers, by
	 Definition \mref{add}.}
\begin{center}  \begin{tabular}{c|rr|rr|rr|r}  %\multispan{3}
\hline $h$ & $ \|u_I-u_h\|_{L^2}$ &$h^n$ &  $ |u_I-u_h|_{H^1}$ & $h^n$ & \#DD
 \\ \hline
$1/4$  & 0.0027120&      & 0.203663 &      &14 \\
$1/12$ & 0.0000716& 1.65 & 0.004456 & 1.74 &14 \\
$1/20$ & 0.0000098& 1.93 & 0.000605 & 1.95 &14 \\
$1/28$ & 0.0000026& 1.97 & 0.000159 & 1.98 &14 \\
$1/36$ & 0.0000009& 1.99 & 0.000058 & 1.99 &14 \\
$1/44$ & 0.0000004& 1.99 & 0.000026 & 1.99 &14 \\
$1/52$ & 0.0000002& 1.99 & 0.000013 & 2.00 &14 \\
      \hline
\end{tabular} \lab{t-1} \end{center} \end{table}

% t=0:0.02:1; plot(t, t-(1-t), t, t-(1-t)/2)
Next, we check our theoretic bounds in Theorem
    \mref{main}.
In \meqref{t1} and \meqref{t2},  if we vary $\theta$
    from $0/7$ to $6/7$,  we can get the
    following theoretic bounds:
    \a{ \frac 77, \frac 57, \frac 37, \frac 17,
      \frac 5{14}, \frac 8{14}, \frac {11}{14}. }
We compute the real bounds for these $\theta$ on various meshes,
  and list them in Table \ref{table2}.
We note that, when $\theta=0/7=0$,  the method is reduced to
   the traditional Robin-Robin DD method 
  (by other researchers, where $\gamma_1=\gamma_2$),
  which converges at a rate of $1-C\sqrt h$, cf. \cite{XQ}.
This can be seen in the first column of Table \mref{table2}.

 \begin{table}[htb]
  \caption{ The reduction rate with different $\theta$ in
   Definition \mref{add}.}
\begin{center}  \lab{table2}
 \begin{tabular}{c|r|r|r|r|r|r|r}  %\multispan{3}
\hline $h\setminus \theta$ & $ 0 $ &$1/7$ &
    $2/7$ &   $3/7$ &  $4/7$ &   $5/7$ &
    $6/7$
 \\ \hline
$1/4$  & 0.764& 0.512& 0.260& 0.096& 0.322& 0.548& 0.774 \\
$1/12$ & 0.865& 0.598& 0.332& 0.115& 0.336& 0.557& 0.779 \\
$1/20$ & 0.894& 0.624& 0.353& 0.116& 0.337& 0.558& 0.779 \\
$1/28$ & 0.910& 0.637& 0.364& 0.116& 0.337& 0.558& 0.779 \\
$1/36$ & 0.920& 0.646& 0.371& 0.116& 0.337& 0.558& 0.779 \\
$1/44$ & 0.927& 0.652& 0.377& 0.116& 0.337& 0.558& 0.779 \\
\hline
$1/72$  & 0.943& 0.665& 0.388& 0.116& 0.337& 0.558& 0.779\\
$1/288$ & 0.971& 0.689& 0.408& 0.126& 0.337& 0.558& 0.779\\
$1/1152$& 0.985& 0.702& 0.418& 0.134& 0.337& 0.558& 0.779\\ \hline
Corollary \mref{co}
        & 1.000& 0.714& 0.428& 0.143& 0.357& 0.571& 0.786 \\
      \hline
\end{tabular}\end{center} \end{table}

Finally, we compare the Robin-Robin domain
    decomposition method  with the traditional
    Dirichlet-Neumann domain decomposition method.
We code directly the
    Dirichlet-Neumann domain decomposition method,
  defined as follows.

\begin{definition} \lab{DN}
 (The Dirichlet-Neumann domain decomposition method.)\\
 Given $w^0 (=0)$ on $\Gamma$, find $u^m\in V_1$,
    $u^m|_\Gamma = w^m$:
      \a{  a_1( u^m, v)
       & =  (f, v)_{\Omega_1}
    \quad \forall v\in V_1\cap H^1_0(\Omega_1).
    } Find $\tilde w^{m+1}\in V_2$:
      \a{ a_2(\tilde w^{m+1},v)  = (f, v)_{\Omega_2}
    -a_1(u^m,v)
    \quad \forall v\in V_2,
    } where $v$ is extended into $\Omega_1$ with 0 nodal values.
    Then
      \a{ w^{m+1} =  \theta w^m +
          (1-\theta) \tilde w^{m+1}.
    }
 \end{definition}

In Table 3, we list the number of Dirichlet-Neumann domain decomposition
     iterations for
   the above test problem, for various $\theta$.
It seems that no matter how to choose $\theta$,  the
    Dirichlet-Neumann domain decomposition method
   (21 iterations) is worse than the new Robin-Robin domain
    decomposition method (14 iterations).

 \begin{table}[htb]
  \caption{ The iteration number for Dirichlet-Neumann DD
    (Definition \ref{DN}.)}
\begin{center} \lab{table3}
 \begin{tabular}{c|r|r|r|r|r|r|r|r}  %\multispan{3}
\hline $h\setminus \theta$ & $ 0 $ &$0.25$ &
    $0.35$ &   $0.4$ &  $0.45$ &   $ 0.5 $ &
    $ 0.55$  &   $ 0.75 $
 \\ \hline
$1/4$   & 88 & 24 & 22 & 25 & 29 & 33 & 38 & 78 \\
$1/12$  &237 & 34 & 21 & 23 & 26 & 30 & 35 & 71 \\
$1/20$  &392 & 37 & 22 & 22 & 25 & 29 & 33 & 68 \\
$1/28$  &548 & 38 & 23 & 21 & 24 & 28 & 32 & 66 \\
$1/36$  &705 & 39 & 23 & 21 & 24 & 27 & 31 & 64 \\
      \hline
\end{tabular}\end{center} \end{table}

\comment{
n=1000; k=sin(pi/4/n)^2; a=4*k^2+4*k; f=sqrt(1+1/a)

t=(2*k+1)*(2*(1+h)-(1-h)*(1-2*k));
s1=sin(pi/(n+1))^2; t1=s1/(a+s1)/(n+1);
2/(n+1)-2*atan(f*tan(pi/(n+1)))/(pi*f)-s1/(a+s1)/(n+1)
t*(1-1/sqrt(1+1/a)+t1)-1+h^2/2

n=112; h=1/(2*n); x=pi/(4*n);
 sb=x^2; ss=x^2-x^4/3;
 sb=sin(x)^2; ss=sb;

(1+3*h*2*(1-h)*sb)/(2*sb+1+sqrt(4*sb^2+4*sb))-1+h^2/2

 left=(3+h/2-2*ss+sb*h)^2*h^2;
 right=(1-h^2/2)^2*4*(ss+1)*ss;
 left-right

n:='n';  h:=1/(2*n); x:=Pi/(4*n);
 sb:=x^2; ss:=x^2-x^4/3;
 left:=(3+h/2-2*ss+sb*h)^2*h^2;
 right:=(1-h^2/2)^2*4*(ss+1)*ss;
g:= expand(-n^10*(left-right));
for i from 0 to 8 do  c[i]:=coeff(g,n,i); evalf(coeff(g,n,i)); od;
n:=2; evalf(n^8*c[8]+n^7*c[7]);
 evalf(n^6*c[6]+n^5*c[5]+n^4*c[4]);
 evalf(n^3*c[3]+ n^1*c[1] );
 evalf( n^2*c[2]+ c[0]);
}

 \bibliographystyle{amsplain}

\begin{thebibliography}{10}

% \bibitem{Achdou}Achdou, Japhet, Le Tallec, Nataf, Rogier, and
%Vidrascu, Domain decomposition methods for nonsymmetric problems, In
%Eleventh International Conference on Domain Decopmsition Methods,
%1999.
%\bibitem{BD}R. E. Bank, and T. Dupont, An optimal order process for
%solving finite element equations, Math. Comp., {\bf 36}(1981),
%35-51.

\bibitem{BF} S. Bennethum, and X. Feng,  A domain decomposition method
  for solving a
  Helmholtz-like  problem in elasticity based on the Wilson
  nonconforming element, M$^2$AN, RAIRO, {\bf 31} (1997), 1-25.

\bibitem{chen1} W. Chen, J. Cheng, M. Yamamoto and W. Zhang,
  The monotone Robin-Robin domain
  decomposition methods for the elliptic problems with
  Stefan-Boltzmann conditions,
Commun. Comput. Phys., {\bf 8} (2010), 642-662.

\bibitem{chen2} W. Chen,  M. Gunzburger, F. Hua and X. Wang,
 A Parallel Robin-Robin Domain Decomposition Method for the
  Stokes-Darcy System,
  SIAM J. Numer. Anal., {\bf 49} (2011), 1064-1084.

\bibitem{ciarlet} P. G. Ciarlet, The Finite Element Method for
  Elliptic Problems, North-Holland, Amsterdam, 1978.


\bibitem{Demmel1997} J.W. Demmel, Applied Numerical Linear Algebra,
   SIAM, 1997.

\bibitem{deng2} Q. Deng, An analysis for a nonoverlapping domain
  decomposition iterative procedure, SIAM J. Sci. Comput., {\bf 18}
  (1997), 1517-1525.

\bibitem{deng3}Q. Deng, A nonoverlapping domain decomposition
  method for nonconforming finite element problems, Comm.  Pure Appl.
  Anal., {\bf 2 } (2003), 295 -306.

\bibitem{Despres} B. Despres, Domain decomposition method and
Helmholtz problem, Mathematical and Numerical Aspects of Wave
  Propagation Phenomena, G. Cohen, L. Halpern and P. Joly, eds.,
  Philadelphia, SIAM, 1991, 44-52.

\bibitem{Discacciati} M. Discacciati, An operator-spliting
approach to non-overlapping domain decomposition methods,  Tech.
Report 14, 2004, Section de Mathmatiques, EPFL.

\bibitem{DQV2007}M. Discacciati, A. Quarteroni and A. Valli,
  Robin-Robin domain
  decomposition methods for the Stokes-Darcy coupling, SIAM J. Numer.
  Anal., {\bf 45} (2007), 1246-1268.

\bibitem{DGG}V. Dolean, M.Gander, L. Gerardo-Giorda,
  Optimized Schwarz methods for Maxwells equations, SIAM J. Sci. Comp.,
   {\bf 31} (2009),  2193-2213.

\bibitem{DH} J. Douglas, and C.S. Huang, Accelerated domain
decomposition iterative procedures for mixed methods based on Robin
transmission conditions, Calcolo, {\bf 35} (1998), 131-147.

\bibitem{DH1} J. Douglas, and C.S. Huang, An accelerated domain
decomposition procedures based on Robin transmission conditions,
BIT, {\bf 37}(1997), 678-686.

\bibitem{Du} Q. Du, Optimization based nonoverlapping domain decomposition
  algorithms and their convergence, SIAM J. Numer. Anal.,{\bf 39} (2001),
  1056-1077.

\bibitem{Feng} X. Feng, Analysis of a domain decomposition method
for the nearly elastic wave equations based on mixed finite element
methods, IMA J. Numer. Anal. {\bf 18} (1998), 229-250.

\bibitem{Nataf1}
E. Flauraud, and F. Nataf, Optimized interface conditions in domain
decomposition methods for problems with extreme contrasts in the
coefficients, J. Comp. Appl. Math, {\bf 189} (2006), 539-554.
%\bibitem{Japhet}C. Japhet, F. Nataf, and F. Rogier, The optimized
%order 2 method, application to convection-diffusion problems, Future
%Generation Computer Systems FUTURE, 18(2001).

\bibitem{Gander01} M. J. Gander, Optimized Schwarz methods,
  SIAM J. Numer. Anal., {\bf 44} (2006), 699-731.

\bibitem{Gander1} M. J. Gander, and G. H. Golub, A nonoverlapping optimized
  Schwarz method which converges with an arbitrarily weak dependence
  on $h$, In fourteenth international conference on domain
  decomposition methods, 2002.

\bibitem{Gander02} M. J. Gander and L. Halpern,
   Optimized Schwarz waveform relaxation for
  advection reaction diffusion problems, SIAM J. Numer. Anal., {\bf
  45}(2007), 666-697.


\bibitem{Gander2} M. J. Gander, L. Halpern, and F. Nataf, Optimized
Schwarz methods, In Twelfth International Conference on Domain
Decomposition Methods, Chiba, Japan, 15-28, 2001.

\bibitem{Gander3} M. J. Gander, F. Magoules, and F. Nataf,
Optimized Schwarz methods without overlap for the Helmholtz
equation, SIAM J. Sci. Comput., {\bf 24} (2002), 38-60.

\bibitem{GH} W. Guo, and L.S. Hou, Generalization and accelerations
  of Lions' nonoverlapping domain decomposition method for linear
  elliptic PDE, SIAM J. Numer. Anal., {\bf 41} (2003), 2056-2080.

\bibitem{Greenbaum1997} A. Greenbaum, Iterative Methods for
  Solving Linear Systems, SIAM, 1997.

\bibitem{lions} P. L. Lions, On the Schwarz alternating method III: a
  variant for nonoverlapping subdomains, in third international
  symposium on domain decomposition methods for partial differential
  equations, T. F. Chan, R. Glowinski, J. Perianx, and O. B. Widlund,
  eds., SIAM, Philadelphia, PA, 1990, 202-223.

\bibitem{Lui1} S. Lui, A Lions non-overlapping domain decomposition method
  for domains with an
 arbitrary interfaces, IMA J. Numer. Anal., {\bf 29} (2009), 332-349.


\bibitem{QX1} L. Qin, and X. Xu, On a parallel Robin-type
   nonoverlapping domain decomposition
   method, SIAM J. Numer. Anal., {\bf 44} (2006), 2539-2558.

\bibitem{QX2} L. Qin, Z. Shi and X. Xu, On the convergence rate of a
parallel nonoverlapping domain decomposition method, Sciences in China,
Series A: Mathematics, {\bf 51} (2008), 1461-1478.

\bibitem{QX3} L. Qin, and X. Xu, Optimized Schwarz methods
  with Robin transmission
  conditions for parabolic problems,
 SIAM J. Sci. Comput., {\bf 31} (2008), 608-623.
%\bibitem{sarkis}M. Sarkis, Nonstandard coarse spaces and Schwarz
%methods for elliptic problems with discontinuous coefficients using
%non-conforming elements, Numer. Math., {\bf 77}(1997), 383-406.

\bibitem{QV} A. Quarteroni, and A. Valli, Domain Decomposition Methods
   for Partial Differential
   Equations, Oxford Science Publications, 1999.

\bibitem{XQ} X. Xu, and L. Qin,
  Spectral analysis of DN operators and optimized Schwarz methods
   with Robin transmission
  conditions, SIAM J. Numer. Anal., {\bf 47} (2010), 4540-4568.

\end{thebibliography}

\end{document}